\documentclass[a4paper,11pt]{article}
\title{Computing points on bielliptic modular curves over fixed quadratic fields}
\author{Philippe Michaud-Jacobs}
\date{\vspace{-4ex}}
\usepackage{amsmath} 
\usepackage{amssymb} 
\usepackage{mathtools}
\usepackage{amsthm}
\usepackage{tikz-cd}

\newtheorem{theorem}{Theorem}[section]
\newtheorem*{theorem*}{Theorem}

\theoremstyle{definition}

\newtheorem{remark}[theorem]{Remark}
\newtheorem*{remark*}{Remark}

\usepackage[numbers]{natbib}
\usepackage{enumerate}
\usepackage{xcolor}

\providecommand{\Q}{\mathbb{Q}}
\providecommand{\Z}{\mathbb{Z}}

\newcommand{\Addresses}{{% additional braces for segregating \footnotesize
  \bigskip
  \footnotesize

 \textsc{Mathematics Institute, University of Warwick, CV4 7AL, United Kingdom}\par\nopagebreak
  \textit{E-mail address}: \texttt{p.rodgers@warwick.ac.uk}
}}

\let\svthefootnote\thefootnote
\newcommand\freefootnote[1]{%
  \let\thefootnote\relax%
  \footnotetext{#1}%
  \let\thefootnote\svthefootnote%
}

\begin{document}

\maketitle

\begin{abstract}
We present a Mordell--Weil sieve that can be used to compute points on certain bielliptic modular curves $X_0(N)$ over fixed quadratic fields. We study $X_0(N)(\mathbb{Q}(\sqrt{d}))$ for $N \in \{ 53,61,65,79,83,89,101,131 \}$ and $\lvert d \rvert < 100$. 
\end{abstract}

%-----------------------------------------------------------
%-----------------------------------------------------------
%-----------------------------------------------------------

\section{Introduction}

%-----------------------------------------------------------
%-----------------------------------------------------------
%-----------------------------------------------------------

There has been a lot of recent interest in computing low-degree points on modular curves, and in particular in computing quadratic points on the curves $X_0(N)$. Computing such points gives much insight into the arithmetic of elliptic curves and has direct applications in the resolution of Diophantine equations (see \cite[p.~888]{small_real} or \cite{MJ_flt} for such examples). \freefootnote{\emph{Date}: \date{\today}.}
\freefootnote{\emph{Keywords}: Modular curve, quadratic points, elliptic curve, Mordell--Weil sieve.}
\freefootnote{\emph{MSC2020}: primary 11G18; secondary 11G05, 14G05.}
\freefootnote{The author is supported by an EPSRC studentship and has previously used the name Philippe Michaud-Rodgers.}

As we range over all quadratic fields, a curve $X_0(N)$ will either have finitely many or infinitely many quadratic points. For those curves $X_0(N)$ that have finitely many quadratic points, these points have been computed in many cases, such as when the genus of $X_0(N)$ is $\leq 5$, or when $X_0(N)$ is bielliptic \cite{Boxquad, ozmansiksek, NajVuk}. If $X_0(N)$ has genus $ \geq 2$ and has infinitely many quadratic points (so that $X_0(N)$ is either hyperelliptic, or bielliptic with an elliptic quotient of positive rank over $\Q$), a geometric description of all the quadratic points has been given in these cases \cite{Boxquad, hyperquad, NajVuk}.

There are precisely $10$ values of $N$ such that the modular curve $X_0(N)$ is bielliptic with an elliptic quotient of positive rank \cite[pp.~26--28]{biellipticbars}. For two of these values of $N$, namely $37$ and $43$, the methods we present will not work (see Remark \ref{fails}), and so we will consider the remaining eight values of $N$, which are \begin{equation*} N \in \mathcal{N} :=  \{53,61,65,79,83,89,101,131 \}. \end{equation*} For each $N \in \mathcal{N}$ the elliptic curve $X_0^+(N) = X_0(N) / w_N$ has rank $1$ over $\Q$. In \cite{Boxquad, NajVuk} it is proven that every quadratic point on $X_0(N)$ arises as the pullback of a rational point on $X_0^+(N)$ (via the natural degree $2$ quotient map). However, this classification does not describe $X_0(N)(K)$ for a given quadratic field $K$. The purpose of this paper is to introduce a Mordell--Weil sieve that can be used to check, for $N \in \mathcal{N}$, whether $X_0(N)(K) = X_0(N)(\Q)$ for a given quadratic field $K$. The sieve uses information on the splitting behaviour of primes in $K$ together with the structure of the Mordell--Weil group of $X_0^+(N)(\Q)$ modulo these primes. The sieve builds on ideas present in the author's work in \cite[pp.~338--340]{MJ_flt}. We prove the following result.

\begingroup
\renewcommand\thetheorem{1}
\begin{theorem}\label{Thm1}
Let $N \in \{53,61,65,79,83,89,101,131\}$ and let $d \in \Z$ such that $\lvert d \rvert < 100$. Then $X_0(N)(\Q(\sqrt{d})) \ne X_0(N)(\Q)$ if and only if $d \in \mathcal{D}_N$, where 
\begin{align*}
\mathcal{D}_{53} & = \{ -43, -11, -7, -1 \}, \\
\mathcal{D}_{61} & = \{ -19, -3, -1, 61 \}, \\
\mathcal{D}_{65} & = \{ -79, -1 \}, \\
\mathcal{D}_{79} & = \{ -43, -7, -3 \}, \\
\mathcal{D}_{83} & = \{ -67, -43, -19, -2 \}, \\
\mathcal{D}_{89} & = \{ -67, -11, -2, -1, 89 \}, \\
\mathcal{D}_{101} & = \{ -43, -19, -1 \}, \\
\mathcal{D}_{131} & = \{ -67, -19, -2 \}.
\end{align*}
\end{theorem}
\endgroup

Although we have considered integers $d$ satisfying $\lvert d \rvert < 100$ here, there are no apparent obstructions to proving analogous results for any integer $d$.

For certain (but not all) integers $d$, the results of Theorem \ref{Thm1} could be achieved by applying \cite[Theorem~1.1]{ozmantwist} or some of the techniques described in \cite{cyclic_isog}. In Section 3, we compare (for $N = 53$) our results with those one can obtain by applying \cite[Theorem~1.1]{ozmantwist}, and use this to provide an example of a curve that violates the Hasse principle. 

We note that results of a similar nature to Theorem \ref{Thm1} (obtained using different techniques) are proven in \cite{hyper} for hyperelliptic curves $X_0(N)$.

The \texttt{Magma} \cite{magma} code used to support the computations in this paper is available at \texttt{https://github.com/michaud-jacobs/bielliptic}.

%-----------------------------------------------------------
%-----------------------------------------------------------
%-----------------------------------------------------------

\section{A Mordell--Weil sieve}

%-----------------------------------------------------------
%-----------------------------------------------------------
%-----------------------------------------------------------

In this section we present a Mordell--Weil sieve and apply it to prove Theorem \ref{Thm1}.

We first describe how to obtain a suitable model for $X_0(N)$ for $N \in \mathcal{N}$. Let $g$ denote the genus of the modular curve $X_0(N)$. We start by computing a basis $f_1, \dots, f_g$ of cusp forms for $S_2(\Gamma_0(N))$ with integer Fourier coefficients such that the Atkin--Lehner involution $w_N$ satisfies $w_N(f_1) = f_1$ and $w_N(f_i) = -f_i$ for $2 \leq i \leq g$ (we refer to such a basis as a \emph{diagonalised basis}). For each $N \in \mathcal{N}$ the curve $X_0(N)$ is non-hyperelliptic of genus $>2$ and we may obtain a nonsingular model for $X_0(N)$ over $\Q$ in $\mathbb{P}^{g-1}_{x_1, \dots, x_g}$ as the image of the canonical embedding on the cusp forms $f_1, \dots, f_g$. The details of this (standard) procedure are described in \cite[pp.~17--38]{galbraith}, and the \texttt{Magma} code we used to do this is adapted from \citep{ozmansiksek}.

The Atkin--Lehner involution $w_N$ on this model is then given by the map $(x_1 : x_2 : \dots : x_g) \mapsto (-x_1 : x_2 : \dots : x_g)$. We denote by $ \psi: X_0(N) \rightarrow X_0^+(N) $ the degree $2$ map induced by quotienting by $w_N$. In each case, we found that the projection map onto the coordinates $x_2, \dots, x_g$ had degree $2$ and image $X_0^+(N)$ (and not some quotient of $X_0^+(N)$), so that the map $\psi$ is given by \begin{align*} \psi : X_0(N) & \longrightarrow X_0^+(N) \\ (x_1 : x_2 : \dots : x_g) & \longmapsto (x_2 : \dots : x_g).  \end{align*} We in fact then obtained a Weierstrass model for $X_0^+(N)$ and composed $\psi$ with this transformation (see the example in Section 3). The reason for using a diagonalised model for $X_0(N)$ is twofold. First, it forces the coordinates of a quadratic point to be of a certain shape, as we see below. Secondly, it greatly speeds up the computations we perform in the sieving step.

Let $K = \Q(\sqrt{d})$ be a quadratic field and write $\sigma$ for the generator of $\mathrm{Gal}(K/\Q)$. Suppose that $P \in X_0(N)(K) \backslash X_0(N)(\Q)$ (equivalently, $P$ is a non-cuspidal quadratic point). In projective coordinates, we may write   \begin{align*} P & = (a_1 + b_1\sqrt{d} : a_2 + b_2 \sqrt{d} :  \dots : a_g + b_g \sqrt{d}), \\
w_N(P) &  = (-a_1 - b_1\sqrt{d} : a_2 + b_2\sqrt{d}  : \dots : a_g + b_g \sqrt{d}), \text{  and} \\ 
P^\sigma & = (a_1 - b_1\sqrt{d} : a_2 - b_2 \sqrt{d} :  \dots : a_g - b_g \sqrt{d}), \end{align*}  where $a_i, b_i \in \Z$ for $1 \leq i \leq g$.

As discussed in the introduction, thanks to the work of Box and Najman--Vukorepa in \cite{Boxquad, NajVuk}, we know that $\psi(P) = \psi(P^\sigma) \in X_0^+(N)(\Q)$, or equivalently, $w_N(P) = P^\sigma$. It follows that \[ P = (b_1 \sqrt{d} : a_2 : \dots : a_g), \] with $b_1 \neq 0$, and we may assume that $\gcd(b_1,a_2, \dots, a_g) = 1$ by rescaling if necessary.

We now present the sieve in the case that $N \neq 65$. In the case $N = 65$ we will need to adapt the sieve slightly, and we discuss this case in the proof of Theorem \ref{Thm1} below. For each $N \neq 65$, we have that $X_0^+(N)(\Q) \cong \Z$, and we let $R$ denote a generator of the Mordell--Weil group, so that\[ \psi(P) = m \cdot R \quad \text{ for some } m \in \Z.\] Let $\ell$ be a prime of good reduction for our models of $X_0(N)$ and $X_0^+(N)$, and consider the following commutative diagram, where $\sim$ denotes reduction mod $\ell$, or a prime of $K$ above $\ell$:
\begin{center}
\begin{tikzcd} X_0(N)\arrow[r, "\psi"] \arrow[d, "\sim"] & X_0^+(N) \arrow[d, "\sim" ] \\  \widetilde{X}_0(N)  \arrow[r,  "\psi_\ell"]& \widetilde{X}_0^+(N)  \end{tikzcd}
\end{center}
By commutativity we have that $\psi_\ell(\widetilde{P}) = m \cdot \widetilde{R}$, so that $\widetilde{P} \in \psi_\ell^*(m \cdot \widetilde{R})$. 
Write $G_\ell$ for the order of $\widetilde{R}$ in the group $\widetilde{X}_0^+(N)(\mathbb{F}_\ell)$. Then \[\psi_l(\widetilde{P}) = m \cdot \widetilde{R} = m_0 \cdot \widetilde{R}\] for some integer $m_0$ satisfying $0 \leq m_0 < G_\ell$ and $m \equiv m_0 \pmod{G_\ell}$. We note that $\psi_\ell^*(m \cdot \widetilde{R}) = \psi_\ell^*(m_0 \cdot \widetilde{R})$. For each integer $0 \leq m_1 < G_\ell$ we explicitly compute the set $\psi_\ell^*(m_1 \cdot \widetilde{R}) \subset \widetilde{X}_0(N)(\mathbb{F}_{\ell^2})$. There are three cases:
\begin{enumerate}[(i)]
\item The set $\psi_\ell^*(m_1 \cdot \widetilde{R})$ consists of a pair of distinct points defined over $\mathbb{F}_\ell$.

If  $\ell$ is inert or ramifies in $K$, then $\psi_\ell^*(m \cdot \widetilde{R})$ will not consist of a pair of distinct points defined over $\mathbb{F}_{\ell}$, and so  $ m \not\equiv m_1 \pmod{G_\ell}$.

\item The set $\psi_\ell^*(m_1 \cdot \widetilde{R})$ consists of a pair of points defined over $\mathbb{F}_{\ell^2}$ (with each point not defined over $\mathbb{F}_\ell$).

If  $\ell$ splits or ramifies in $K$, then $\psi_\ell^*(m \cdot \widetilde{R})$ will not consist of a pair of points defined over $\mathbb{F}_{\ell^2}$, and so  $ m \not\equiv m_1 \pmod{G_\ell}$.

\item The set $\psi_\ell^*(m_1 \cdot \widetilde{R})$ consists of a single point defined over $\mathbb{F}_\ell$.
\end{enumerate}
Verifying the splitting behaviour of the prime $\ell$ in cases (i) and (ii) leaves us with a list of possible values for $m \pmod{G_\ell}$. 

We may then repeat this process with a list of primes $\ell_1, \dots, \ell_s$. For each $1 \leq i \leq s$ we obtain a list of possibilities for $m \pmod{G_{\ell_i}}$. This gives a system of congruences that we may solve using the Chinese remainder theorem to obtain a list of possibilities for $m \pmod{\mathrm{lcm}(G_{\ell_i})_{1 \leq i \leq s)}}$. If no solution exists to this system of congruences then we obtain a contradiction and conclude that $X_0(N)(K) = X_0(N)(\Q)$.

\begin{proof}[Proof of Theorem \ref{Thm1} for $N \ne 65$]
Let $N \in \mathcal{N} \backslash \{65\}$. We start by proving that if $d \in \mathcal{D}_N$ then $X_0(N)(\Q(\sqrt{d})) \ne X_0(N)(\Q)$. We computed the preimages $\psi^*(t \cdot R)$ for $t \in \Z$ with $\lvert t \rvert \leq 5$ and verified the field of definition of the points we obtained. For each $d \in \mathcal{D}_N$ we found a pair of quadratic points in $X_0(N)(\Q(\sqrt{d}))$.

For the converse, we suppose that $\lvert d \rvert < 100$ with $d \notin \mathcal{D}_N$ and aim to prove that $X_0(N)(\Q(\sqrt{d})) = X_0(N)(\Q)$. We note that this is immediate if $\Q(\sqrt{d}) = \Q$, so we assume that $\Q(\sqrt{d})$ is a quadratic field. Suppose, for a contradiction, that $P \in X_0(N)(\Q(\sqrt{d})) \backslash X_0(N)(\Q)$. We applied the Mordell--Weil sieve described above with the following (ordered) choice of primes (we discuss this choice in Remark \ref{choice}): \begin{equation}\label{L} \mathcal{L} = \{ \ell \mid d : \ell \nmid 2N \} \cup \{ \ell < 1000 : \ell \nmid 2N \text{ and } q \nmid G_\ell \text{ for primes } q > 7 \}. \end{equation} In each case this led to a contradiction.
\end{proof}

\begin{proof}[Proof of Theorem \ref{Thm1} for $N = 65$]
The proof of the theorem in this case is very similar to the case $N \neq 65$. The key difference is that $X_0^+(N)(\Q) \cong \Z \oplus \Z / 2\Z$. We write $Q$ for the $2$-torsion point and choose a point $R$ such that any point in $X_0^+(N)(\Q)$ may be expressed as $m \cdot R + n \cdot Q$ for some $m \in \Z$ and $n = 0$ or $1$. For our choice of $R$, we found that $\psi^*(-R)$ and $\psi^*(-2R)$ consisted of pairs of quadratic points defined over $\Q(\sqrt{-1})$ and $\Q(\sqrt{-79})$ respectively, proving one direction of the theorem.

For the converse, let $d \notin \mathcal{D}_N$ be such that $\Q(\sqrt{d})$ is a quadratic field and $\lvert d \rvert < 100$. Suppose, for a contradiction, that there exists a point $P \in X_0(N)(\Q(\sqrt{d})) \backslash X_0(N)(\Q)$. Either $\psi(P) = m \cdot R$ or $m \cdot R + Q$. In the first case, we apply the sieve exactly as in the proof for $N \neq 65$ (with the same choice of primes) to achieve a contradiction. In the second case, we again apply the sieve in the same way, except that we work with the point $m \cdot R + Q$ instead. To be precise, for each prime $\ell$, we have that $\widetilde{P} \in \psi_\ell^*(m \cdot \widetilde{R} + \widetilde{Q})$, and so we compute $\psi_\ell^*(m_1 \cdot \widetilde{R} + \widetilde{Q})$ for each $0 \leq m_1 < G_\ell$. By considering each preimage and the splitting behaviour of $\ell$ in the quadratic field $\Q(\sqrt{d})$ we obtain a list of possibilities for $m \pmod {G_\ell}$. As in the previous case, we achieved a contradiction for each $d$.
\end{proof}

The total computation time for the proof of Theorem \ref{Thm1} was 2500 seconds running on a 2200MHz AMD Opteron.

\begin{remark}\label{choice} We discuss the choice of primes $\mathcal{L}$ used in the proof of the theorem. We start by choosing the primes that ramify as these usually eliminate the greatest number of possibilities for $m \pmod{G_\ell}$. We then choose primes $\ell$ such that the values $G_\ell$ are small and share many prime factors. There are two reasons for doing this. First, when solving each system of congruences we are more likely to obtain fewer solutions, and ultimately a contradiction. Secondly, we avoid (or reduce the likelihood) of a combinatorial explosion, since the lowest common multiple of the $G_\ell$ can grow very quickly if the primes $\ell$ are not chosen carefully. We note that the largest prime $\ell$ we in fact ended up reaching was $\ell = 593$ in the case $N = 101$ and $d = 31$. 
\end{remark}

\begin{remark} \label{fails}
As discussed in the introduction, we have not considered the curves $X_0(37)$ or $X_0(43)$. The curve $X_0(37)$ is bielliptic with an elliptic quotient of positive rank, but it is also hyperelliptic, and therefore has two sources of infinitely many quadratic points, meaning the sieve we have presented would not work. The reason the sieve does not work for the curve $X_0(43)$ is due to the fact that $X_0(43)$ has a non-cuspidal rational point that is fixed by the Atkin--Lehner involution $w_{43}$. The sieve cannot rule out the possibility that a quadratic point is equal to this non-cuspidal rational point.
\end{remark}

Although we have presented this sieve for certain specific bielliptic modular curves $X_0(N)$, the sieve could be suitably adapted to compute quadratic points on a wider range of curves. Indeed, it should even be possible to apply a similar sieve to compute quadratic points on any curve $X$ with a degree $2$ quotient of genus $\geq 1$, if there are finitely many quadratic points on $X$ not arising as pullbacks of rational points on this quotient, and that these have all been computed. Although, as in the case $X_0(43)$ discussed above, there may be obstructions to the sieving process succeeding.

%-----------------------------------------------------------
%-----------------------------------------------------------
%-----------------------------------------------------------

\section{Example computations}

%-----------------------------------------------------------
%-----------------------------------------------------------
%-----------------------------------------------------------

In this section we provide some details of computations in the case $N = 53$. We start by obtaining a model for the genus $4$ curve $X_0(53)$ on which the Atkin--Lehner involution acts diagonally. By searching for relations between a diagonalised basis of cusp forms we obtain the following model in $\mathbb{P}^3_{x_1,x_2,x_3,x_4}$:
\begin{align*}
& x_1^2 - x_2^2 + 2x_2x_3 - 6x_2x_4 + 11x_3^2 - 6x_3x_4 - x_4^2 = 0, \\
& x_1^3 - x_1x_2^2 + 2x_1x_2x_3 - 6x_1x_2x_4 + 11x_1x_3^2 - 
    6x_1x_3x_4 - x_1x_4^2 = 0, \\
& x_1^2x_2 - x_2^3 + 2x_2^2x_3 + 5x_2x_3^2 + 5x_2x_4^2 - 
    6x_3^2x_4 + 6x_4^3 = 0, \\
& x_1^2x_3 - x_2^2x_3 + 2x_2x_3^2 - 6x_2x_3x_4 + 11x_3^3 - 
    6x_3^2x_4 - x_3x_4^2 = 0, \\
& x_1^2x_4 - x_2x_3^2 + 3x_2x_3x_4 - 5x_2x_4^2 + 10x_3^2x_4 - 
    6x_3x_4^2 = 0, \\
& x_2^2x_4 - x_2x_3^2 + x_2x_3x_4 + x_2x_4^2 - x_3^2x_4 + x_4^3 = 0. 
\end{align*}
The equations for the Atkin--Lehner involution  on this  model are given by $w_{53} : (x_1 : x_2 : x_3 : x_4) \mapsto (-x_1 : x_2 : x_3 : x_4)$.
The degree $2$ map to the elliptic curve $X_0^+(53)$ is then simply given by the projection map $(x_1 : x_2 : x_3 : x_4) \mapsto (x_2 : x_3 : x_4)$. We then apply a transformation to take the image of this projection map to the following Weierstrass form:
\[ X_0^+(53): \; Y^2Z+XYZ + YZ^2 = X^3-X^2Z. \] The map $\psi$ is given by \begin{align*} \psi: X_0(53) & \longrightarrow X_0^+(53) \\ (x_1 : x_2 : x_3 : x_4) & \longmapsto (x_2x_3 + x_3x_4 : x_2^2 + x_2x_4 - x_3x_4 + x_4^2 :  x_3x_4).
\end{align*} 
We choose $R = (0 : -1 : 1) \in X_0^+(53)(\Q)$ as a generator of the Mordell--Weil group.

We now exhibit some steps in the sieving process for $d = -47$. As in the previous section, we will assume that $P \in X_0(53)(\Q(\sqrt{d})) \backslash X_0(53)(\Q)$ and write $\psi(P) = m \cdot R$. We apply the sieve with the primes $ \ell = 5,7,$ and $11$. The prime $\ell = 5$ is inert in $\Q(\sqrt{d})$ and we find that $m \equiv 3$ or $5 \pmod{6}$. Next,  $\ell = 7$ splits and we find that $m \equiv 0, 3, 4, 7,$ or $11 \pmod{12}$. Combining this with the previous condition we have $m \equiv 3$ or $11 \pmod{12}$. Finally, the prime $\ell = 11$ is inert and we find that $m \equiv 1, 2, 5, 7,$ or $10 \pmod{12}$, a contradiction. We conclude that  $X_0(53)(\Q(\sqrt{-47})) = X_0(53)(\Q)$.

We have in fact proven that $X_0(53)(K) = X_0(53)(\Q)$ for any quadratic field $K$ in which $5$ and $11$ are inert and $7$ splits. In a similar vein, when $d = 3$ we achieved a contradiction using only the prime $\ell = 3$, and this proves that $X_0(53)(K) = X_0(53)(\Q)$ for any quadratic field $K$ in which $3$ ramifies. This type of result is similar to those appearing in \cite{hyper}, and we could seek to prove more results along these lines, but we do not pursue this here.

In order to verify that the sieve is working as expected, we can try applying it for a value $d \in \mathcal{D}_{53}$. For example, applying the sieve with $d = -11 \in \mathcal{D}_{53}$ and the primes in  $\mathcal{L}$ (defined as in (\ref{L})) outputs a list of possibilities for $m \pmod{\mathrm{lcm}(G_\ell)_{\ell \in \mathcal{L}}}$, where $\mathrm{lcm}(G_\ell)_{\ell \in \mathcal{L}} =  63504000$. We find that either $m = 1$ or that $m \ge 1905121$. The fact that $m = 1$ remains as a possibility is because $\psi^*(1 \cdot R)$ consists of a pair of quadratic points defined over $\Q(\sqrt{-11})$.

It is interesting to consider how the results of Theorem \ref{Thm1} overlap with the results one may obtain by applying the techniques of \cite{ozmantwist}, which give a criterion for testing whether $X_0(N)(\Q(\sqrt{d})) = X_0(N)(\Q)$ by checking local points on the curve $X_0^{(d)}(N)$. This curve is the quadratic twist of $X_0(N)$ by the Atkin--Lehner involution $w_N$ and the quadratic extension $\Q(\sqrt{d}) / \Q$ (see \cite[p.~628]{hasse} for a precise definition). For $N = 53$, in Theorem \ref{Thm1} we prove that $X_0(N)(\Q(\sqrt{d})) = X_0(N)(\Q)$ for $117$ values of $d$, with $d$ squarefree and $\lvert d \rvert < 100$. Applying \cite[Theorem~1.1]{ozmantwist},\footnote{As discussed in \cite[p.~39]{Banwait}, some care is needed in order to interpret correctly parts (5) and (6) of \cite[Theorem~1.1]{ozmantwist}.} we reproduced these results for $94$ of these values. For some values of $d$, where our sieving method works, but applying \cite[Theorem~1.1]{ozmantwist} fails, we can often obtain examples of curves that violate the Hasse principle. Continuing our example with $d = -47$, we find that $X_0^{(-47)}(53)(\Q)$ has points everywhere locally by applying \cite[Theorem~1.1]{ozmantwist}. However, $X_0^{(-47)}(53)(\Q) = \emptyset$, as proven above. Similar examples are considered in \cite[pp.~344--346]{ozmantwist}, where a standard Mordell--Weil sieve is applied to twists of a hyperelliptic curve.

%-----------------------------------------------------------
%-----------------------------------------------------------
%-----------------------------------------------------------

\bibliographystyle{plainnat}

\Addresses

\end{document}